\documentclass[12 pt]{amsart}
\usepackage{amssymb}
  
\newcommand{\gothic}{\mathfrak}

\newcommand{\m}{{\gothic{m}}}

\newcommand{\Hom}{\operatorname{Hom{}}}

\newcommand{\Tor}{\operatorname{Tor{}}}

\renewcommand{\phi}{\varphi}

\newtheorem{thm}{Theorem}

\newtheorem{prop}[thm]{Proposition}

\begin{document}
                                                                                                                       
\title{Ideals defining Gorenstein rings are (almost) never products}
                                                                                                                       
\author{Craig Huneke}
                                                                                                                       
\address{Department of Mathematics \\
        University of Kansas \\
        Lawrence, KS
        66045}
                                                                                                                       
\email{huneke@math.ku.edu}
                                                                                                                       
\urladdr{http://www.math.ku.edu/\textasciitilde huneke}
                                                                                                                       
\begin{abstract} This note proves that if $S$ is an unramified regular local ring and
$I,J$ proper ideals of height at least two, then $S/IJ$ is never Gorenstein.
\end{abstract}

\date{\today}

                                                                                   

\thanks{ The author gratefully acknowledges support on NSF grant
DMS-0244405. I also thank Bill Heinzer for correspondence concerning the
paper, and in particular for sending me the statement and argument of
Proposition 1.}
                                                                                                                       
\maketitle
                                                                                                                       
\bibliographystyle{amsplain}

This paper answers in the unramified case a question that was asked by Eisenbud and Herzog: can an
ideal in a regular local ring that defines a Gorenstein quotient ring ever be the product of two proper ideals, except
in the obvious case in which a principal ideal is a product? 

A related (unpublished) result was proved by W. Heinzer, D. Lantz and
K. Shah in the early 1990s. They proved that an ideal generated by
a system of parameters in a Gorenstein ring can never be the product of two proper ideals.
Explicitly, they prove the following:

\begin{prop} Let $(R,\m)$ be a Gorenstein local ring of
dimension at least two. 
Then an ideal I generated by a
system of parameters never the product of two proper ideals.
\end{prop}

\begin{proof} (Heinzer, Lantz, Shah) Without loss of generality we may
assume the residue
field of $R$ is infinite. Assume that  $I = JK$. We claim that we may assume $J$ and
$K$ are also generated by systems of parameters.
For let $J'$ and $K'$ be minimal reductions of $J$
and $K$. Since the residue field of $R$ is infinite, both $J'$ and $K'$ are generated
by systems of parameters. 
Then $I = JK$ is integral over $J'K'$ since they
have the same extension to any valuation
overring.  Since $I$ is not integral over any
proper subideal, we must have $I = J'K'$. So
we may assume $J$ and $K$ are generated by systems
of parameters. In particular, $J$ has a unique minimal
overideal. Hence by Matlis duality the
annihilator $L/I$ of $J/I$ has a unique minimal
under ideal in $R/I$. Thus $L/I = a(R/I)$ for any
$a \in L$ not in that underideal. Since $I = JK$
is contained in $JL$ which is contained in $I$,
we have $I = JL = J(aR + I) = aJ + JI$. But this
means $I = aJ$, so I is not $\m$-primary.
\end{proof}

Our main result is the following theorem.

\begin{thm} Let $(S,\m)$ be an unramified regular local ring, and let $I$ be an ideal of
height at least two. If $S/I$ is Gorenstein, then $I\ne JK$ for any two proper
ideals $J,K$.
\end{thm}

\begin{proof} By way of contradiction we assume that $I = JK$, where $J$ and $K$ are proper ideals.
We reduce to the case in which $I$ is $\m$-primary using a simple
reduction suggested by the referee, which is easier than what this author originally
did. First we make the residue field infinite in case it is not already infinite by
replacing $S$ by $S(t) = S[t]_{\m S[t]}$, where $t$ is a variable. This does not
change the assumption that $S$ is an unramified regular local ring, and $IS(t) = JS(t)\cdot KS(t)$. 
We will prove that the height of $IS(t)$ is one, which will be a contradiction.
Henceforth we assume that the residue field of $S$ is infinite.

We next reduce to the case that $I$ is $\m$-primary. 
Choose a maximal regular sequence $x_1,...,x_d$ on $S/I$
consisting of elements whose images in $\m/(\m^2 +I)$ are independent. We can furthermore
assume that if $S$ is mixed characteristic $p$, then $p\notin (x_1,...,x_d)  +  \m^2$ by
choosing $x_1,...,x_d$ generally enough, unless $\m = (x_1,...,x_d)$. In this latter
case, $I = 0$, a contradiction.
One can then replace $S$ by
$S' = S/(x_1,...,x_d)$ which is still an unramified regular local ring, and replace $I$ by
$I' = IS'$. In this case, $S'/I'$ is still Gorenstein, $I'$ is still a product, and $I'$ is
primary for the maximal ideal. Moreover the height of $I$ and the height of $I'$ are the same.

We have reduced to the case in which $I$ is $\m$-primary. Among all representations of $I$
as a product, $I = JK$ for two proper ideals $J$ and $K$, choose $J$ and $K$
each maximal with respect to this property. Since $I = JK\subseteq (I:K)K\subseteq I$,
the maximality shows that $I:K = J$ and similarly $I:J = K$. Set $R = S/I$. 
For an arbitrary $R$-module $M$, we let $\lambda(M)$ denote the length of $M$.
From duality, using the fact that $R$ is Gorenstein, we see that
$$
\lambda(R/(J+K)R) = \lambda(\Hom_R(R/(J+K)R,R)) = \lambda((0:(J+K)R)$$ 
 $$= \lambda((0:JR)\cap (0:KR))  = \lambda((K\cap J)R).
$$

Lifting these ideals back to $S$, we obtain that $\lambda(S/(J+K)) = \lambda((J\cap K)/JK)$,
i.e. that $$\lambda(\Tor_0^S(S/J,S/K)) = \lambda(\Tor_1^S(S/J,S/K)).$$
Since both of these modules have finite length, $\chi(S/J,S/K) = 0$, where
whenever $M$ and $N$ are finitely generated $S$-modules such that $M\otimes_SN$ has
finite length we define
$$\chi(M,N) = \sum_{i=0}^{\infty} (-1)^i\lambda(\Tor_i^S(M,N)).$$

Hence $$\sum_{i=2}^{\infty} (-1)^i\lambda(\Tor_i^S(S/J,S/K)) = 0.$$
By \cite{L} (see also \cite{H}), this forces $\Tor_i^S(S/J,S/K) = 0$ for all $i\geq 2$.
We claim this forces the dimension of $S$ to be at most one: let
$\bf F$ be the minimal free resolution of $S/J$. After tensoring with $S/K$, the last map in
the resulting complex can never be injective since every socle element of $S/K$ must go 
to zero. Since the last free module in the minimal resolution of $S/J$ occurs at
the dimension of $S$, the vanishing of all higher Tors past $2$ forces the dimension
to be one. Hence the height of $I$ is one, a contradiction. 
\end{proof}

Theorem 2 has several possible extensions. Of course, to begin with, one would like to know the
result in the ramified case. The most reasonable related question seems to be the
following:

\medskip

Question: \it Let $S$ be a regular local ring and $I$ an ideal of height $c$. Assume that
$I = JK$ for two proper ideals $J,K$, and that $S/I$ is Cohen-Macaulay. Then is the
type of $S/I$ at least $c$\rm?

\medskip

Perhaps there is an elementary argument that answers this question in the affirmative, which would simplify and
generalize the result above. For example, in the case $I$ is $\m$-primary and
$J = \m$, the result holds since the type of $S/I$ is just the socle dimension, and
if $I = \m K$, then $K/\m K$ is contained in the socle. Hence the type is at least
the minimal number of generators of $K$, which is at least the height of $I$ by
Krull's height theorem. At the very least this special case shows that some homological algebra
dealing with Krull's theorem might be needed.

\end{document}